\documentclass{amsart}
\usepackage{tabu}
\usepackage{amssymb} 
\usepackage{amsmath} 
\usepackage{amscd}
\usepackage{amsbsy}
\usepackage{comment, enumerate}
\usepackage[matrix,arrow]{xy}
\usepackage{hyperref}
\usepackage{mathrsfs}
\usepackage{color}
\usepackage{mathtools,caption}

\newcommand{\QQ}{\mathbb Q}
\newcommand{\ZZ}{\mathbb Z}

\newcommand{\Fp}{\mathbb F_p}

\newcommand{\barFp}{\overline{\mathbb F}_p}

\newcommand{\fp}{\mathfrak p}

\newcommand{\fl}{\mathfrak l}

\newcommand{\sqth}{\sqrt{-3}}

\newcommand{\GQ}{G_{\QQ}}
\newcommand{\phiE}{\hat\phi_{E,p}}
\newcommand{\brphiE}{\phi_{E,p}}
\newcommand{\rhoE}{\hat\rho_{E,p}}
\newcommand{\brhoE}{\rho_{E,p}}
\newcommand{\brhof}{\rho_{f,p}}
\newcommand{\brhoA}{\rho_{A_\beta,\pi}}
\newcommand{\bound}{137}

\DeclareMathOperator{\Gal}{Gal}
\DeclareMathOperator{\Res}{Res}
\DeclareMathOperator{\GL}{GL}
\DeclareMathOperator{\PGL}{PGL}
\DeclareMathOperator{\Norm}{Norm}
\DeclareMathOperator{\Tr}{Tr}
\DeclareMathOperator{\Hom}{Hom}
\DeclareMathOperator{\rank}{rank}
\DeclareMathOperator{\Frob}{Frob}

\newtheorem{thm}{Theorem}
\newtheorem{lem}{Lemma}[section]
\newtheorem{prop}[lem]{Proposition}

\newtheorem{cor}[lem]{Corollary}

\theoremstyle{definition}
\theoremstyle{remark}

\title{On the generalized Fermat equation $a^2+3b^6=c^n$}

\author{Angelos Koutsianas}
\address{Department of Mathematics, The University of British Columbia, 1984 Mathematics Road
Vancouver, BC, Canada}
\email{akoutsianas@math.ubc.ca}

\date{\today}


\begin{document}

\begin{abstract}
In this paper, we prove that the only primitive solutions of the equation $a^2+3b^6=c^n$ for $n\geq 3$ are $(a,b,c,n)=(\pm 47,\pm 2,\pm 7,4)$. Our proof is based on the modularity of Galois representations of $\QQ$-curves and the work of Ellenberg \cite{Ellenberg04} for big values of $n$ and a variety of techniques for small $n$.
\end{abstract}

\maketitle

\section{Introduction}

The remarkable breakthrough of Andrew Wiles about the proof of Taniyama-Shimura conjecture which leaded to the proof of Fermat's Last Theorem introduced a new and very rich area of modern number theory. A variety of techniques and ideas have been developed for solving the generalized Fermat equation of the form
\begin{equation}\label{eq:generalizedFermat}
Aa^p + Bb^q  = Cc^r.
\end{equation}
Because the literature is very rich we refer to \cite{BennettChenDahmenYazdani15} for a detailed exposition of the cases of \eqref{eq:generalizedFermat} that have been solved. In this paper we prove the following

\begin{thm}\label{thm:main}
Let $n\geq 3$ be an integer. The only primitive solution of equation
\begin{equation}\label{eq:main}
a^2 + 3b^6 = c^n
\end{equation}
is $(a,b,c,n)=(\pm 47,\pm 2,\pm 7,4)$. A solution $(a,b,c)$  is called primitive if $a,b,c$ are pairwise coprime integers and $ab\neq 0$.
\end{thm}

For the proof of Theorem \ref{thm:main} we use the recent proof of modularity of $\QQ$-curves as a result of the proof of Serre's modularity conjecture \cite{KhareWintenberger09a,KhareWintenberger09b,Kisin09} and the study of the arithmetic of $\QQ$-curves by many mathematicians \cite{Quer00,Ellenberg04,Ribet04}.  Even though we are not able to give a detailed proof it seems that for the equation $a^2 + db^6 = c^n$ and fix $d>0$ we are able to attach a Frey $\QQ$-curve only for the cases $d=1$ \cite{BennettChen12} and $3$, which makes these values special.

The paper is organised as follows. In Section \ref{sec:preliminaries} we recall the terminology and theory of $\QQ$-curves. In Section \ref{sec:Frey_curve} we introduce a Frey curve which we prove it is a $\QQ$-curve and we study its arithmetic properties. In Section \ref{sec:proof} we prove Theorem \ref{thm:main} when $n\geq 11$ is a prime using Ellenberg's analytic method \cite{Ellenberg04} which we explain in Section \ref{sec:eliminate_CM}. In Section \ref{sec:small_exponents} we prove Theorem \ref{thm:main} for the small exponents $n=3,4,5,7$. Finally, in Appendix \ref{sec:appendix} we compute the rational points of the curve $Y^2=X^6+48$ which we need for the case $n=4$.

The computations of the paper were performed in \textbf{Magma} \cite{Magma} and the programs can be found in author's homepage \url{https://sites.google.com/site/angeloskoutsianas/}.

\section{Preliminaries}\label{sec:preliminaries}

In this section we recall the main definitions of the $\QQ$-curves and their attached representations; we recommend \cite{BennettChen12}, \cite{EllenbergSkinner01}, \cite{Quer00} and \cite{Ribet04} for a more detailed exposition. 

Let $K$ be a number field and $E/K$ be an elliptic curve without CM such that for every $\sigma\in G_\QQ$ there exists an isogeny $\mu_E(\sigma):{}^{\sigma}E\longmapsto E$. Then $E$ is called a \textit{$\QQ$-curve defined over $K$}. We make a choice of the isogenies above such that $\mu_E$ is locally constant.

Let
\begin{equation}
c_E(\sigma,\tau) = \mu_E(\sigma){}^\sigma\mu(\tau)\mu(\sigma\tau)^{-1},\in\left(\Hom(E,E)\otimes_\ZZ\QQ\right)^*=\QQ^*
\end{equation}
where $\mu_E^{-1}:=(1/\deg\mu_E)\mu_E^\vee$ and $\mu_E^\vee$ is the dual of $\mu_E$. Thus $c_E$ determines a class in $H^{2}(G_\QQ,\QQ^*)$ which depends only on the $\overline\QQ$-isogeny class of $E$. Tate has showed that $H^{2}(G_\QQ,\QQ^*)$ is trivial when $G_\QQ$ acts trivially on $\QQ^*$. So, there exists a continuous map $\beta:G_\QQ\rightarrow\QQ^*$ such that
\begin{equation}
c_E(\sigma,\tau) = \beta(\sigma)\beta(\tau)\beta(\sigma\tau)^{-1}
\end{equation}
The map $\beta$ is called a \textit{splitting map} of $c_E$.

We define an action of $G_\QQ$ on $\overline\QQ_p\otimes_{\ZZ_p}T_p E$ given by
\begin{equation}
\rhoE(\sigma)(1\otimes x)= \beta(\sigma)^{-1}\otimes\mu(\sigma)(\sigma(x))
\end{equation}
From the definition of $\rhoE$ we have that $\mathbb P\rhoE\mid_{G_K}\simeq \mathbb P\phiE$ where 
\begin{equation}\label{eq:phi_rep}
\phiE:\Gal(\bar K/K)\rightarrow \GL_2(\ZZ_p)
\end{equation}
is the usual Galois representation attached to the $p$-adic Tate module of $E$ (see \cite[Proposition 2.3]{EllenbergSkinner01}). Given a splitting map $\beta$, Ribets \cite{Ribet04} attaches an abelian variety $A_\beta$ over $\QQ$ of $\GL_2$-type such that $E$ is a simple factor over $\overline\QQ$.

From the definition of $\rhoE$ we understand that the representation depends on $\beta$. Let $M_\beta$ be the field generated by the values of $\beta$. We want to make a choice of $\beta$ such that it factors over a number field of low degree and $c_E(\sigma,\tau) = \beta(\sigma)\beta(\tau)\beta(\sigma\tau)^{-1}$ as elements in $H^{2}(G_\QQ,\overline\QQ^*)$. Then we choose a twist $E_\beta/K_\beta$ such that $c_{E_\beta}(\sigma,\tau)=\beta(\sigma)\beta(\tau)\beta(\sigma\tau)^{-1}$ as cocycles and let $K_\beta$ be the field over $\beta$ factors which is called the \textit{splitting field of $\beta$}. In this case, the abelian variety $A_\beta$ is a quotient of $\Res_{K_\beta/\QQ}E_\beta$ over $\QQ$. The endomorphism algebra of $A_\beta$ is equal to $M_\beta$ and the representation on the $\pi^n$-torsion points of $A_\beta$ coincides with the representation $\rhoE$ above, where $\pi$ is a prime ideal in $M_\beta$ above $p$.

Finally, we define the $\epsilon:G_\QQ\rightarrow\overline\QQ^*$ given by
\begin{equation}
\epsilon(\sigma)=\frac{\beta(\sigma)^2}{\deg\mu(\sigma)}
\end{equation}
Then, $\epsilon$ is a character such that
\begin{equation}
\det(\rhoE)=\epsilon^{-1}\cdot \chi_p
\end{equation}
where $\chi_p$ is the the $p$-th cyclotomic character. We can attach a residual representation associate to $\rhoE$ (see \cite[p. 107]{EllenbergSkinner01})
\begin{equation}
\brhoE:\Gal(\bar\QQ/\QQ)\rightarrow\barFp^*\GL_2(\Fp).
\end{equation}
Similarly, we denote by $\brphiE$ the residual representation associate to $\phiE$.

\section{Frey $\QQ$-curve attached to $a^2 + 3b^6 = c^p$}\label{sec:Frey_curve}

In this section we attach a Frey $\QQ$-curve over $K = \QQ(\sqth)$ to a primitive solution $(a,b,c)$ of \eqref{eq:main}. Let $n=p$ be an odd prime. We define


\begin{equation}\label{eq:Q_curve}
E_{a,b}: Y^2 = X^3 -9\sqth b(4a - 5\sqth b^3) X + 18(2a^2 - 14\sqth ab^3 - 33 b^6)
\end{equation}
When it is not confusing we use the notation $E$ instead of $E_{a,b}$. The invariants of $E$ are given by
\begin{align}
j(E) & = 2^4\cdot 3^3\cdot\sqth\cdot b^3\cdot\frac{(4a-5\sqth b^3)^3}{(a+\sqth b^3)^3\cdot (a-\sqth b^3)},\label{eq:jinvariant} \\
\Delta(E) & = -2^8\cdot 3^7\cdot (a - \sqth b^3)\cdot(a + \sqth b)^3,\label{eq:Delta}\\
c_4(E) & = 2^4 \cdot 3^3\cdot\sqth\cdot b\cdot(4a-5\sqth b^3),\label{eq:c4}\\
c_6(E) & = -2^6\cdot 3^5\cdot (2a^2-14\sqth b^3a - 33 b^6).
\end{align}

\begin{lem}\label{lem:rational_j}
Let $a/b^3\in\mathbb P^1(\QQ)$. Then the $j$-invariant of $E$ lies in $\QQ$ only when
\begin{itemize}
\item $a/b^3=0$ and $j=54000$, or
\item $a/b^3=\infty$ and $j=0$.
\end{itemize}
\end{lem}

\begin{proof}
From \eqref{eq:jinvariant} and for $a/b^3=\infty$ we have that $j=0$. Let assume that $a/b^3\neq\infty$. After cleaning denominators of \eqref{eq:jinvariant} and taking real and imaginary parts using the restriction that $j,a/b^3\in\QQ$ we end up with
\begin{align*}
-A^4j^\prime + 720A^2 + 9j^\prime - 1125 & =0\\
(-A^2j^\prime + 32A^2 - 3j^\prime - 450)A & = 0 
\end{align*}
where $j^\prime = j/432$ and $A = a/b^3$. From the second equation we have that either $A = 0$ or $j^\prime=\frac{32A^2 - 450}{A^2 + 3}$. For $A= 0$ we have the first case of the lemma. Replacing $j^\prime$ to the first equation above we end up with 
\begin{equation}
-32A^4 + 1266A^2 - 2475 = 0
\end{equation}
which we can easily check that does not have any solution over $\QQ$.
\end{proof}

\begin{lem}\label{lem:cm_j}
The curve $E$ does not have complex multiplication unless
\begin{itemize}
\item $a/b^3=0$, $j=54000$ and $d(\mathcal O) = -12$ or
\item $a/b^3=\infty$, $j=0$ and $d(\mathcal O) = -3$.
\end{itemize}
\end{lem}

\begin{proof}
Let assume that $E$ has complex multiplication. Then from the theory of complex multiplication we know that the $j(E)$ is a real algebraic number. Because $j(E)\in\QQ(\sqth)$ we conclude that $j(E)\in\QQ$. Because the list of $j$-invariants of elliptic curves with complex multiplication with $j\in\QQ$ it is known (see \cite{Cox89}) we have the result.
\end{proof}

\begin{lem}\label{lem:2_3_nmid_c}
Let $(a,b,c)$ be a primitive solution of \eqref{eq:main}, then $c$ is divisible by a prime different from $2$ and $3$. 
\end{lem}

\begin{proof}
Because $(a,b,c)$ is a solution of $a^2 + 3b^6=c^p$ we have that $3\nmid c$. Because $p\geq 3$ and $a^2+3b^6\not\equiv 0\mod 8$ we have that $2\nmid c$.
\end{proof}

Because of Lemma \ref{lem:cm_j} we assume that $E$ has no complex multiplication. The curve $E$ is a $\QQ$-curve because it is $3$-isogenous to its conjugate and the isogeny is defined over $K$ (see \textit{IsQcurve.m}). We make a choice of isogenies $\mu(\sigma):{}^{\sigma}E\longmapsto E$ such that $\mu(\sigma)=1$ for $\sigma\in G_{K}$ and $\mu(\sigma)$ equal to the $3$-isogeny above for $\sigma\not\in G_K$. 

Let $d$ be the \textit{degree map} (see \cite{Quer00}), then we have that $d(\GQ)=\{1,3\}\subset \QQ^*/\QQ^{*2}$. The fixed field $K_d$ of the kernel of the degree map is $\QQ(\sqth)$. Then $(a,d)=(-3,3)$ is a dual basis in the terminology of \cite{Quer00}. We can see that $(-3,3)$ is unramified and so $\epsilon = 1$, $K_\epsilon=\QQ$ and $K_\beta = \QQ(\sqth)$. Moreover, we have $\beta(\sigma)=\sqrt{d(\sigma)}$ and so $M_\beta=\QQ(\sqrt{3})$.

Let $A_\beta=\Res_{K/\QQ}E$. Since $K_\beta=K$ we understand that $\xi_K(E)$ has trivial Schur class. Thus from \cite[Theorem 5.4]{Quer00} we have that $A_\beta$ is a $\GL_2$-type variety with $\QQ$-endomorphism algebra isomorphic to $M_\beta$.

Let $\fp_2$ and $\fp_3$ be the primes in $K$ above $2$ and $3$ respectively.

\begin{lem}\label{lem:conductor_E3}
The elliptic curve $E$ is a minimal model with conductor equal to\footnote{For some of the computations it is more convenient to use the isomorphic to $E$ curve 
$$E^\prime : Y^2 + 6\sqth bXY - 12(\sqth b^3 + a) Y = X^3.$$} 
\begin{equation}\label{eq:conductor_E3}
N(E) = \fp_2^2\cdot\fp_3^8\cdot\prod_{\fp\mid c}\fp.
\end{equation}
\end{lem}

\begin{proof}
Let assume that $\fp$ is a prime in $K$ that does not divide $2,3$. Then from \eqref{eq:Delta} and \eqref{eq:c4} we understand that $E$ has multiplicative reduction at $\fp$.

Let $\fp_3$ be the prime in $K$ above $3$. From Tate's algorithm we can prove that $E$ has $IV^*$ reduction type and because $v_{\fp_2}(\Delta)=14$ we have the exponent for $\fp_3$.

Let $\fp_2$ be the prime in $K$ above $2$. Because $p\geq 3$ we have that $2\nmid c$, Lemma \ref{lem:2_3_nmid_c}. So, we have 

$$(v_{\fp(2)}(c_4),v_{\fp(2)}(c_6),v_{\fp(2)}(\Delta))=
\begin{cases}
(\geq 7,7,8) & \text{if }v_2(b)>0,\\
(4,6,8) & \text{otherwise.}
\end{cases}$$
From \cite[Tableau IV]{Papadopoulos93} we conclude that $E$ has $I_0^*$, $I_1^*$ or $IV^*$ reduction type. Applying Tate's algorithm we can show that $E$ has neither $I_0^*$ nor $I_1^*$ reduction type.
\end{proof}

\begin{lem}\label{lem:conductor_A}
The conductor of $A_\beta$ is 
\begin{equation}
d^2_{K/\QQ}\cdot \Norm_{K/\QQ}(N(E))=2^{4}\cdot 3^{10}\cdot \prod_{p\mid c}p^2.
\end{equation}
\end{lem}

\begin{proof}
This is an immediate consequence of \cite[Proposition 1]{Milne72} and the fact that $K$ is unramified outside $3$.
\end{proof}

Since $A_\beta$ is of $\GL_2$-type with $M_\beta=\QQ(\sqrt{3})$, the conductor $N_{A_\beta}$of the system of $M_{\beta,\pi}[G_\QQ]$-modules $\left\lbrace\widehat V_\pi(A_\beta)\right\rbrace$ is given by
\begin{equation}\label{eq:conductor_system}
N_{A_\beta} = 2^2\cdot 3^5\cdot \prod_{p\mid c}p
\end{equation}
as it is explained in \cite{Chen10} where $M_{\beta,\pi}$ is the completion of $M_\beta$ with respect to $\pi$. In the following lines we compute the Serre invariants $N_\rho=N(\brhoE)$, $k_\rho = k(\brhoE)$ and  $\epsilon_\rho = \epsilon(\brhoE)$.

\begin{prop}
The representation $\brphiE\mid_{I_p}$ is finite flat for $p\neq 2,3$.
\end{prop}

\begin{proof}
Let $\fp$ be a prime above $p$. By Lemma \ref{lem:conductor_E3} we know that $E$ has good or multiplicative reduction at $\fp$. In the case of multiplicative reduction the exponent of $\fp$ in the minimal discriminant of $E$ is divisible by $p$. Finally, $K$ is only ramified at $3$ and so $I_p\subseteq G_K$.
\end{proof}

\begin{prop}
The representation $\brphiE\mid_{I_\ell}$ is trivial for $\ell\neq 2,3,p$.
\end{prop}

\begin{proof}
Let $\fl$ be a prime above $\ell$. Because of  Lemma \ref{lem:conductor_E3} we know that $E$ has good or multiplicative reduction at $\fl$. In the case of multiplicative reduction the exponent of $\fl$ in the minimal discriminant of $E$ is divisible by $p$. Finally, $K$ is only ramified at $3$ and so $I_\ell\subseteq G_K$.
\end{proof}

\begin{prop}\label{prop:rep_conductor}
Suppose $p\neq 2,3$. Then $N_\rho = 972$.
\end{prop}

\begin{proof}
Because we want to compute the Artin conductor of $\brhoE$, we consider only ramification at primes above $\ell\neq p$.

Let consider $\ell\neq 2,3,p$. We recall that $K=K_\beta$. Because $\ell\neq 3$ we have that $K_\beta$ is unramified at $\ell$, so $I_\ell\subset G_{K}$. Because $\brhoE|_{G_K}\simeq\brphiE$ and $\brphiE\mid_{I_\ell}$ is trivial we have that $\brhoE$ is trivial at $I_\ell$. Thus, $\brhoE$ is unramified outside $2,3,p$.

Suppose $\ell=2,3$. From \eqref{eq:jinvariant} we understand that $E$ has potential good reduction at primes above $2,3$. That means that $\phiE|_{I_\ell}$ factors through a finite group of order divisible only by $2,3$. Thus, $\rhoE|_{I_\ell}$ factors through a finite group of order divisible only by $2,3$. It follows that the exponent of $\ell$ in the conductor of $\brhoE$ is the same as in the conductor of $\rhoE$ as $p\neq 2, 3$.
\end{proof}

\begin{prop}\label{prop:rep_weight}
Suppose $p\neq 2,3$. Then $k_\rho = 2$.
\end{prop}

\begin{proof}
The weight is determined by $\brhoE|_{I_p}$. For $p\neq 3$ we have that $K$ is unramified at $p$ and so $I_p\subset G_K$. Because $\brhoE|_{G_K} \simeq \brphiE$, $\brphiE\mid_{I_p}$ is finite flat and the determinant of $\brphiE$ is the cyclotomic $p$-th character then from \cite[Prop. 4]{Serre87} we have the conclusion.
\end{proof}

\begin{prop}
The character $\epsilon_\rho$ is trivial.
\end{prop}

\begin{proof}
This is a consequence of the fact that $\epsilon$ is trivial and the properties of $\rhoE$.
\end{proof}

From \cite[Proposition 3.2]{Ellenberg04} and Lemma \ref{lem:2_3_nmid_c} we have

\begin{prop}\label{prop:reducible_potential}
Let assume that $\brhoE$ is reducible for $p\neq 2,3,5,7,13$. Then $E$ has potentially good reduction at all primes above $\ell > 3$.
\end{prop}

An immediate consequence of Proposition \ref{prop:reducible_potential} and Lemma \ref{lem:2_3_nmid_c} is the following.

\begin{cor}\label{cor:irreducible_rep}
The representation $\brhoE$ is irreducible for $p\neq 2,3,5,7,13$.
\end{cor}

\begin{prop}
If $p=13$, then $\brhoE$ is irreducible.
\end{prop}

\begin{proof}
This is similar to \cite[Proposition 17]{BennettChen12} which is based on results in \cite{Kenku79} about $\QQ$-rational points on $X_0(39)/w_3$.
\end{proof}

\section{Proof of Theorem \ref{thm:main}}\label{sec:proof}

\begin{proof}
Let assume that $p\geq 11$ be an odd prime. Let $(a,b,c)$ be a primitive solution of \eqref{eq:main}. We attach to $(a,b,c)$ the curve $E$. Because of the modularity of $\QQ$-curves which follows from Serre's conjecture \cite{KhareWintenberger09a,KhareWintenberger09b,Kisin09}, the Ribet's level lowering \cite{Ribet90} and the results in Section \ref{sec:Frey_curve} we have that there exists a newform $f\in S_2(\Gamma_0(972))$ such that $\brhoE\simeq\brhof$.

There are $7$ newforms of level $972$. Four of them are rational\footnote{Let $f$ be a newform and $K_f$ the eigenvalues field of $f$. Then we say that $f$ is \textit{rational} when $K_f=\QQ$ and \textit{irrational} when $K_f\neq\QQ$.} with complex multiplication by $\QQ(\sqth)$ and the other three are irrational. In Section \ref{sec:eliminate_CM} we show how we can prove that non-solutions arise from the rational newforms for $p\geq 11$ using Ellenberg's analytic method, see Proposition \ref{prop:no_solutions_rational_newforms_Ellenberg}. For the irrational newforms we use Proposition \ref{prop:congruence_criterion} and we prove that $p\leq 7$ (see \textit{CongruenceCriterion.m}).
\end{proof}

\begin{prop}\label{prop:congruence_criterion}
Let $f\in S_2(\Gamma_0(972))$ and $p,q$ be primes such that $p\geq 11$, $q\geq 5$ and $q\neq p$. We define
$$
B(q,f)=
\begin{cases}
N(a_q(E) - a_q(f)) & \text{if }a^2+3b^6\not\equiv 0\mod q\text{ and }\left(\frac{-3}{q}\right)=1,\\
N(a_q(f)^2-a_{q^2}(E)-2q) &\text{if }a^2+3b^6\not\equiv 0\mod q\text{ and }\left(\frac{-3}{q}\right)=-1,\\
N((q+1)^2 - a_q(f)^2) & \text{if } a^2+3b^6\equiv 0\mod q.
\end{cases} 
$$
where $a_{q^i}(E)$ is the trace of $\Frob_q^i$ acting on the Tate module $T_p(E)$. Then $p|B(q,f)$.
\end{prop}

\begin{proof}
From Section \ref{sec:Frey_curve} we recall that $A_\beta = \Res_{K/\QQ}(E)$ and $M_\beta=\QQ(\sqrt{3})$. Let $\pi$ be a prime of $M_\beta$ above $p$. As we mentioned in Section \ref{sec:preliminaries} we have that $\brhoA=\brhoE$ where $\brhoA$ is the mod $\pi$ representation of $G_\QQ$ on the $\pi^n$-torsion points of $A_\beta$. We recall that 
\begin{equation}
\brhoE(\sigma)(1\otimes x)= \beta(\sigma)^{-1}\otimes\mu(\sigma)(\brphiE(\sigma)(x))
\end{equation}
where $\brphiE$ is the representation of $G_K$ acting on $T_p(E)$ and $1\otimes x\in M_{\beta,\pi}\otimes T_p(E)$. We also recall that $\brhoA=\brhoE\simeq\brhof$ and $\beta(\sigma)=\sqrt{d(\sigma)}$.

Let assume the case $a^2+3b^6\equiv 0\mod q$. By \eqref{eq:conductor_system} we have that $q\parallel N_{A_\beta}$ and from \cite[Th\'{e}or\`{e}m (A)]{Carayol86}, \cite[Theorem 3.1]{DarmonDiamondTaylor97} we have that
\begin{equation}
p\mid N(a_q(f)^2 - (q+1)^2).
\end{equation} 

For the rest of the proof we assume that $a^2 + 3b^6\not\equiv 0\mod q$. When $\left(\frac{-3}{q}\right)=1$ we have that $\sigma=\Frob_q\in G_K$ and $\mu(\sigma)=1$, $d(\sigma)=1$, so $\Tr\brhoA(\sigma) = \Tr\brphiE(\sigma)$. Because $\brhoA=\brhoE\simeq\brhof$ we conclude that $a_q(E)\equiv a_q(f)\mod\pi$ and so $p\mid N(a_q(E) - a_q(f))$.

Suppose $\left(\frac{-3}{q}\right)=-1$, then $\sigma=\Frob_q\not\in G_K$. Because $\sigma^2\in G_K$ and similarly to the above lines we have that $\Tr\brhoA(\sigma^2) = \Tr\brphiE(\sigma^2)=a_{q^2}(E)$. We know that
\begin{equation}
\frac{1}{\det(I - \brhoA(\sigma)q^{-s})}=\exp\sum_{n=1}^\infty\Tr\brhoA(\sigma^n)\frac{q^{-ns}}{n}=\frac{1}{1-\Tr\brhoA(\sigma)q^{-s}+qq^{-2s}}
\end{equation}
From the coefficient of $q^{-2s}$ we have that $\Tr\brhoA(\sigma^2)=\Tr\brhoA(\sigma)^2 - 2q$. As above we conclude that $a_q(f)^2\equiv a_{q^2}(E) + 2q\mod\pi$, so $p\mid N(a_q(f)^2-a_{q^2}(E)-2q)$.
\end{proof}

\section{Eliminating the CM forms}\label{sec:eliminate_CM}

In this section we explain and apply the method of Ellenberg \cite{Ellenberg04} which allows us to prove that no solutions of \eqref{eq:main} arise from the rational newforms for $p\geq 11$.

\begin{prop}[Proposition 3.4 \cite{Ellenberg04}]\label{prop:not_in_split_normalizer}
Let $K$ be an imaginary quadratic field and $E/K$ a $\QQ$-curve of squarefree  degree $d$. Suppose the image of $\mathbb P\brhoE$ lies in the normalizer of a split Cartan subgroup of $\PGL_2(\Fp)$, for $p=11$ or $p>13$ with $(p,d)=1$. Then $E$ has potentially good reduction at all primes of $K$ not dividing $6$.
\end{prop}

\begin{prop}[Proposition 3.6 \cite{Ellenberg04}]\label{prop:not_in_nonsplit_normalizer}
Let $K$ be an imaginary quadratic field and $E/K$ a $\QQ$-curve of squarefree  degree $d$. Then there exists a constant $M_{K,d}$ such that if the image of $\mathbb P\brhoE$ lies in the normalizer of a nonsplit Cartan subgroup of $\PGL_2(\Fp)$ and $p>M_{K,d}$ then $E$ has potential good reduction at all primes of $K$.
\end{prop}

The constant $M_{K,d}$ can be chosen to be a lower bound of the primes Proposition \ref{prop:newform_finite_Af} holds.

\begin{prop}[Proposition 3.9 \cite{Ellenberg04}]\label{prop:newform_finite_Af}
Let $K$ be an imaginary quadratic field and $\chi_K$ be the associate Dirichlet character. Then for all but finitely many primes $p$, there exists a weight $2$ cusp form $f$, which is either
\begin{itemize}
\item a newform in $S_2(\Gamma(dp^2))$ with $w_pf=f$ and $w_df=-f$,
\item a newform in $S_2(\Gamma(d^\prime p^2))$ with $d^\prime$ a proper divisor of $d$ and $w_pf=f$
\end{itemize}
such that $A_{f\otimes\chi }(\QQ)$ is a finite group.
\end{prop}

The reasons why Proposition \ref{prop:newform_finite_Af} implies Proposition \ref{prop:not_in_nonsplit_normalizer} are explained in \cite[p. 775]{Ellenberg04}. Before we show how we can prove when Proposition \ref{prop:newform_finite_Af} holds we need to introduce some notation.

Let $f$ be a modular form with $q$-expansion
\begin{equation}
f=\sum_{m=0}^\infty a_m(f)q^n.
\end{equation}
We define $L_\chi(f):=L(f\otimes\chi,1)$ where $\chi$ is a Dirichlet character. We can think $a_m$ and $L_\chi$ as linear functions in the space of modular forms.

Moreover, we denote by $\mathcal F$ a Petersson-orthonormal basis for $S_2(\Gamma_0(N))$ and we define
\begin{equation}
(a_m,L_\chi)_N := \sum_{f\in\mathcal F}a_m(f)L_\chi(f)
\end{equation}
For $M\mid N$ we denote by $(a_m,L_\chi)_N^M$ the contribution to $(a_m,L_\chi)_N$ of the forms which are new at level $M$. We also define
\begin{equation}
(a_m,L_\chi)_{p^2}^{p-\text{new}} := (a_m,L_\chi)_{p^2}-(a_m,L_\chi)_{p^2}^p.
\end{equation}
In \cite{Ellenberg04} it is explained that Proposition \ref{prop:newform_finite_Af} holds as long as $|(a_1,L_\chi)_{p^2}^{p-\text{new}} |>0$.

In our case we have $d = 3$, $\chi_{-3}=\left(\frac{-3}{n}\right)$ and $q=3$. So we have the following.

\begin{prop}\label{prop:p_newform}
Let $p\geq 11$ be a prime. Then there exists a newform $f\in S_2(\Gamma_0(p^2))$ such that $w_pf=f$ and $L(f\otimes \chi_{-3})\neq 0$.
\end{prop}

\begin{proof}
In \cite[Lemma8]{DieulefaitUrroz09} the authors prove that $|(a_1,L_\chi)^{p-\text{new}}|>0$ for $p\geq\bound$. For $p < \bound$ we have written a Magma program which proves that the same it true for $11\leq p < \bound$ (see \textit{NewformTwist.m}).
\end{proof}

\begin{prop}\label{prop:no_solutions_rational_newforms_Ellenberg}
Let $p\geq 11$ be a prime. Then primitive solutions of \eqref{eq:main} do not arise from a rational newform $f\in S_2(\Gamma_0(972))$.
\end{prop}

\begin{proof}
Let $f$ be a rational newform of $S_2(\Gamma_0(972))$. Then we know that $f$ has complex multiplication and so the image of $\brhof$ lies in the normalizer of a Cartan group. Because of Lemma \ref{lem:2_3_nmid_c} there exists a prime in $K$ not above $6$ such that $E$ does not has potential good reduction. Because of Propositions \ref{prop:not_in_split_normalizer}, \ref{prop:not_in_nonsplit_normalizer} and \ref{prop:p_newform} we have that $\brhoE$ does not lie in the normalizer of a Cartan group for $p\neq 13$. However, this is a contradiction to the fact that $\brhoE\simeq\brhof$.

For $p=13$ we have problem only for the split case which we can not exclude using Proposition \ref{prop:not_in_split_normalizer}. However, the argument following \cite[Proposition 3.9]{Ellenberg04} also works for the split case (see also \cite[Proposition 6]{BennettEllenbergNg10}). So, from Proposition \ref{prop:p_newform} we have the result.
\end{proof}

\section{Solutions for the remaining small exponents}\label{sec:small_exponents}

In this final section we finish the proof of Theorem \ref{thm:main} proving that \eqref{eq:main} has no primitive solutions for $n=3,4,5,7$. We need the following lemma.

\begin{lem}\label{lem:parametrization_for_p}
Let $p\geq 5$ an odd prime and $x,y,z$ pairwise coprime integers such that $x^2 +3y^2 = z^p$. We define 
\begin{align}
f_1(u,v) & = \sum_{i=0}^{\frac{p-1}{2}}\binom{p}{2i+1}(-3)^{\frac{p-1}{2}-i}u^{2i+1}v^{p-1-2i}\\
f_2(u,v) & = \sum_{i=0}^{\frac{p-1}{2}}\binom{p}{2i}(-3)^{\frac{p-1}{2}-i}u^{2i}v^{p-2i}
\end{align}
Then there exist integers $u_0,v_0$ with $(u_0,v_0)=1$ such that $x = f_1(u_0,v_0)$, $y = f_2(u_0,v_0)$ and $z = u_0^2 + 3v_0^2$.
\end{lem}

\begin{proof}
This is a consequence of factoring $x^2 +3y^2 = z^p$ over the ring of integers of $\QQ(\sqth)$.
\end{proof}

\subsection{Case $n=3$:}
Let assume $n=3$, then a solution with $b\neq 0$ corresponds to a rational point of the elliptic curve $E:y^2 = x^3-3$ via the equation
\begin{equation}
\left(\frac{a}{b^3}\right)^2 = \left(\frac{c}{b^2}\right)^3-3.
\end{equation}
The curve $E$ is Cremona's label 972B1 with trivial Mordell-Weil group \cite{Cremona97}.

\subsection{Case $n=4$:}
Let assume that $n=4$. We know that $2\mid b$. For the parametrization of the conic $X^2 + 3Y^2=1$ we have that there exist coprime $x,y\in\ZZ$ such that
\begin{align}
\left\{\begin{aligned}
\frac{a}{c^2}=&\frac{3x^2-y^2}{3x^2+y^2}\\
\frac{b^3}{c^2}=&\frac{-2xy}{3x^2+y^2}
\end{aligned}\right.
\end{align}

Because $a,c$ are odd we understand that there exists $k\geq 0$ such that
\begin{align}
\left\{\begin{aligned}
a &= \frac{3x^2-y^2}{2^k}\\
c^2 & = \frac{3x^2+y^2}{2^k}\\
b^3 & = \frac{-2xy}{2^k}
\end{aligned}\right.
\end{align}

\begin{lem}\label{lem:k_zero}
Let $a,b,c,x,y$ as above. Then $k=0$.
\end{lem}

\begin{proof}
Let assume that $k>0$. Because $a$ is odd we have that $x,y$ are odd too. Since $3x^2-y^2\equiv 2\mod 4$ we have that $k=1$. Then $3x^2+y^2\equiv 0\mod 4$ and so $2\mid c$ which is a contradiction.
\end{proof}

Because $c$ is odd and Lemma \ref{lem:k_zero} we have that $2\nmid y$. So, we conclude that there are coprime integers $b_1,b_2$ such that $x=4b_1^3$ and $y=b_2^3$. Thus we have that 
\begin{equation}
c^2 = 48b_1^6 + b_2^6
\end{equation}
Because $b_1\neq 0$ the point $(\frac{b_2}{b_1},\frac{c}{b_1^3})$ is a rational point on the genus $2$ curve
\begin{equation}
C:Y^2 = X^6 + 48
\end{equation}
Unfortunately, the Jacobian of $C$ has rank $2$ and classical Chabauty method does not work. However, $C$ is bielliptic and we are able to apply the ideas in \cite{FlynnWetherell99}. In the Appendix \ref{sec:appendix} we prove the following

\begin{prop}\label{prop:covering_chabauty}
The set of rational points of $C$ is $C(\QQ)=\left\lbrace\infty^{\pm},(\pm 1,\pm 7)\right\rbrace$.
\end{prop}

From $C(\QQ)$ it is easy to compute the solutions of \eqref{eq:main} for $n=4$.

\subsection{Case $n=5$:}
From Lemma \ref{lem:parametrization_for_p} we have that there exist coprime integers $u,v$ such that $b^3=f_2(u,v)=v(5u^4-30u^2v^2+9v^4)$. Thus we can conclude that there exist coprime $b_1,b_2$ such that
\begin{align*}
&\left\{
\begin{aligned}
v=&5^2\cdot b_1^3\\
5u^4-30u^2v^2+9v^4=&5\cdot b_2^3\\
\end{aligned}\right.
& & \text{or} & 
\left\{\begin{aligned}
v=&b_1^3\\
5u^4-30u^2v^2+9v^4=& b_2^3\\
\end{aligned}\right.
\end{align*}
For the first case we have that
\begin{equation}
(u^2+\sqth v^2)^2-2^2\cdot 3^2\cdot 5^7\cdot b_1^{12}=b_2^3.
\end{equation}
Then the point $(\frac{b_2}{5^2\cdot b_1^4},\frac{u^2+\sqth v^2}{5^3\cdot b_1^6})$ is a $\QQ(\sqth)$-point of the elliptic curve $E:Y^2=X^3+180$. However, using Magma we can prove that $E(\QQ(\sqth))$ is trivial which is a contradiction.

For the second case we have 
\begin{equation}\label{eq:W}
5u^4-30u^2b_1^6+9b_1^{12}=W_1^2-20u^4=b_2^3.
\end{equation}
where $W_1=3b_1^6 - 5u^2$. Firstly, we consider the case $(u,b_1)\equiv (1,1)\mod 2$. Then we understand that there exists odd $W_1^\prime$ such that $W_1=2W_1^\prime$. Thus, we have 
\begin{equation}
W_1^{\prime 2}-5u^2 = 2b_2^{\prime 3}
\end{equation}
where $b_2=2b_2^\prime$. Taking the last equation modulo $4$ we understand that $2\mid b_2^\prime$, thus $W_1^{\prime 2}-5u^2\equiv 0\mod 8$ which is a contradiction to the fact that both $W_1^\prime$ and $u$ are odd. 

Let assume now that one of the $b_1$ and $u$ is even\footnote{This case is the same like the second case of equation (12) in \cite{BennettChen12}.}. Then we deduce $W_1$ is coprime to $10$. Factoring \eqref{eq:W} over $\QQ(\sqrt{5})$, which has class number $1$, we have that there exist $m$ and $n$ both odd or even such that
\begin{equation}\label{eq:W_u_over_Q5}
W_1+2\sqrt{5}u^2 = \left(\frac{1+\sqrt{5}}{2}\right)^e \left(\frac{m+n\sqrt{5}}{2}\right)^3.
\end{equation}
where $e=0,1,2$. 

For the case $e=1$ and expanding \eqref{eq:W_u_over_Q5} we have that 
\begin{align*}
m^3+15m^2n+15mn^2+25n^3 =& 16W_1\\
m^3+3m^2n+15mn^2+5n^3 =& 32u^2
\end{align*}
Subtracting the last two equations we get $3m^2n+5n^3=4W_1-8u^2$. Because $m$ and $n$ are both odd or even we deduce $3m^2n+5n^3\equiv 0\mod 8$ while $4W_1-8u^2\equiv 4 \mod 8$ which is a contradiction.

For $e=2$ we have
\begin{align*}
3m^3+15m^2n+45mn^2+25n^3 =& 16W_1\\
m^3+9m^2n+15mn^2+15n^3 =& 32u^2
\end{align*}
From the last two equations we have $3m^2n + 5n^3 = 24u^2 - 4W_1$. As before we have that $3m^2n + 5n^3\equiv 0\mod 8$ while $24u^2 - 4W_1\equiv 4\mod 8$ which is a contradiction.

Finally, we have the case $e=0$. It holds
\begin{align}
m(m^2 + 15n^2) = & 8W_1=8(3b_1^6 - 5u^2)\label{eq:first_mn}\\
n(3m^2 + 5n^2) = & 16u^2\label{eq:second_mn}
\end{align}
From the last two equations we have
\begin{equation}\label{eq:mnb1}
48b_1^6 = (m + 5n)(2m^2 + 5mn + 5n^2)
\end{equation}
Because $\gcd(m,n)\mid 2$ and from \eqref{eq:second_mn} we have that $n=2^{e_1}3^{e_2}n_1^2$ for some integer $n_1\in\ZZ$ and $e_i\in\{0,1\}$. Moreover, if we consider \eqref{eq:second_mn} modulo $5$ we understand that $(e_1,e_2)=(1,0)$ or $(0,1)$. For the case, $(e_1,e_2)=(1,0)$ we have that $n=2n_1^2$. Because $m\equiv n\mod 2$ we have $m=2m_1$ and equations \eqref{eq:first_mn} and \eqref{eq:second_mn} become
\begin{align}
m_1(m_1^2 + 15n_1^4) = & 3b_1^6 - 5u^2\label{eq:first_m1n1}\\
n_1^2(3m_1^2 + 5n_1^4) = & 2u^2\label{eq:second_m1n1}
\end{align}
From \eqref{eq:first_m1n1} we conclude that $m_1$ is odd since one of $b_1,u$ is even. As long as $m_1$ is odd we also understand from \eqref{eq:first_m1n1} that $n_1$ is even. However, from \eqref{eq:second_m1n1} we have that $2v_2(n_1)=1 + 2v_2(u)$ which is a contradiction.

Let assume now that $(e_1,e_2)=(0,1)$ and so $n=3n_1^2$. From \eqref{eq:mnb1} we understand that $3\mid m$ which is a contradiction to the fact that $m,n$ are coprime away from $2$.

\subsection{Case $n=7$:}
Finally, let assume that $n=7$. From Lemma \ref{lem:parametrization_for_p} we have that there exist coprime integers $u,v$ such that $b^3=f_2(u,v)=v(7u^6 - 105u^4v^2 + 189u^2v^4 - 27v^6)$. Thus we can conclude that there exist coprime $b_1,b_2$ such that
\begin{align*}
&\left\{
\begin{aligned}
v=&7^2b_1^3\\
7u^6 - 105u^4v^2 + 189u^2v^4 - 27v^6=&7b_2^3\\
\end{aligned}\right.
& & \text{or} & 
\left\{\begin{aligned}
v= & b_1^3\\
7u^6 - 105u^4v^2 + 189u^2v^4 - 27v^6 = & b_2^3\\
\end{aligned}\right.
\end{align*}
We define $f=7u^6 - 105u^4v^2 + 189u^2v^4 - 27v^6$. From the theory of invariants of cubic binary forms (see \cite{Cremona99} or \cite{DahmenThesis}) we have that $28h^3 = g^2 + 27f^2$ where
\begin{align}
h & = 7u^4 - 18u^2v^2 + 27v^4\\
g & = 91u^6 - 189u^4v^2 - 567u^2v^4 + 729v^6.
\end{align}

Let $M=\QQ(\sqth)$ and $\omega=\frac{1+\sqth}{2}$. Then it holds $28h^3 = (g + 3\sqth f)(g-3\sqth f)$.

\begin{lem}\label{lem:S_primitive_n7_c1}
Let $S_M=\{\fp\subset\mathcal O_M:\fp|2,3,7\}$. The triple $(g,f,h)$ is $S_M$-primitive and there exist $z_1,z_2\in M$ and $d_1,d_2\in M(S_M,3)$ with $d_1d_2/28\in M^{*3}$ such that
\begin{align}
g + 3\sqth f &= d_1z_1^3\\
g - 3\sqth f &= d_2z_2^3
\end{align}
\end{lem}

\begin{proof}
Because of \cite[Lemma 3.1]{Bruin03} it is enough to prove that $(g,f,h)$ is $S_M$-primitive. Because $g,f,h\in\QQ$ it is enough to prove that $p\nmid\gcd(g,f,h)$ for $p\neq 2,3,7$. 

Let assume that there exists a prime $p$ that divides $f$, $g$ and $h$. It holds
\begin{align*}
\Res(f,g;u) = & 2^{42}\cdot 3^{18}\cdot 7^6\cdot v^{36}\\
\Res(f,g;v) = & 2^{42}\cdot 3^{18}\cdot 7^6\cdot u^{36}
\end{align*}
Because $p$ has to divide both $\Res(f,g;u)$ and $\Res(f,g;v)$, $p\neq 2,3,7$ and $(u,v)=1$ we have the result.
\end{proof}

Because $f=b_2^3$, $7b_2^3$ and from Lemma \ref{lem:S_primitive_n7_c1} we have that 
\begin{align*}
g + 3a\sqth b_2^3 &= d_1z_1^3\\
g - 3a\sqth b_2^3 &= d_2z_2^3
\end{align*}
where $a=1,7$. Subtracting the above two equation we have the following

\begin{prop}
With the notation as above we have that $(z_1,z_2,b_2)$ corresponds to a point on the cubic form
\begin{equation}\label{eq:elliptc_curve_n7_c1}
C:6a\sqth X_3^3 = d1X_1^3 - d_2X_2^3.
\end{equation}
where $a=7,1$ and $(X_1,X_2,X_3)\in\mathbb P^2(M)$.
\end{prop}

Using Magma we can find a degree $9$ birational map $\phi_C$ defined over $M$ from $C$ to the Jacobian $E_C$ of $C$ which is an elliptic curve defined over $M$. Again using Magma we prove that $E_C$ has zero rank for any choice of $d_1$, $d_2$ and $a$. 

Let $P=(a_1,a_2,a_3)\in\mathbb P^2(M)$ be point on $C$ which lies in the preimage of $E_C(M)$ with $a_3=1,0$. Then there exists $\lambda\in M$ such that $z_1=\lambda a_1$, $z_2=\lambda a_2$ and $b_2=\lambda a_3$. For the case $a_3=0$ we conclude that $f=b_2=0$ which means that $b=0$ in \eqref{eq:main}. Let assume that $a_3=1$, so $b_2=\lambda$. Because $b_2,g\in\ZZ$ we have that $g= (d_1a_1^3-3a\sqth)b_2^3\in\QQ$. But using Magma we prove that this never happens and so \eqref{eq:main} has no solutions for $n=7$ (see \textit{Exponent7.m}).

\section{Appendix}\label{sec:appendix}

In this section we prove Proposition \ref{prop:covering_chabauty} applying the ideas of Flynn and Wetherell \cite{FlynnWetherell99} and the elliptic curve Chabauty \cite{Bruin03}. 

We recall that $C: Y^2 = X^6 + 48$. We define
\begin{equation}
E : y^2 = x^3 + 48
\end{equation}
It holds $E(\QQ)=\ZZ$ and the generator is $(1,7)$. Let $K=\QQ(a)$ where $a^3 + 48=0$ and $\{0,(1,7)\}$ be a set of representatives of $E(\QQ)/2E(\QQ)$. According to \cite[Lemma 1.1(a)]{FlynnWetherell99} the square of the $X$-coordinate of a rational point of $C$ is the $x$-coordinate of one of the two elliptic curves,
\begin{align}
E_1 : &~y^2 = x(x^2 + ax + a^2)\\
E_2 : &~y^2 = (1-a)x(x^2 + ax + a^2)
\end{align}
For both curves have $\rank E_i(K)<3$, so we can apply elliptic curve Chabauty \cite{Bruin03} (see also \cite{BremnerTzanakis04}, \cite{FlynnWetherell99}) to compute $E_i(K)\cap E_i(\QQ)$. Writing a Magma script (see\footnote{In \textit{Exponent4.m} we make the change of variables $(x,y)=(X/(1-a),Y/(1-a))$ for $E_2$ to bring the curve in the standard Weierstrass form.} \textit{Exponent4.m}) we prove the following,

\begin{prop}
It holds,
\begin{align*}
E_1(K)\cap E_1(\QQ) = & \{\infty,(0,0)\}\\
E_2(K)\cap E_2(\QQ) = & \{\infty,(0,0),(1\pm 7)\}
\end{align*}
\end{prop}
Then we can easily prove that $C(\QQ)=\left\lbrace\infty^{\pm},(\pm 1,\pm 7)\right\rbrace$.

\section*{Acknowledgement}
The author would like to thank Professor John Cremona for providing access to the servers of the Number Theory Group of Warwick Mathematics Institute where all the computations took place.

\bibliographystyle{alpha}
\bibliography{on_the_generalized_Fermat_equation_a2_3b6_cn}

\end{document}